\documentclass[11pt]{huber_article}

\usepackage{amsmath,amssymb,amsthm}
\usepackage{booktabs}
\usepackage{dsfont}

\newcommand{\prob}{\mathbb{P}}
\newcommand{\unifdist}{\textsf{Unif}}

\newcommand{\ind}{{\mathds{1}}}
\newcommand{\mean}{\mathbb{E}}

\newcommand{\sd}{\operatorname{SD}}
\newcommand{\median}{\operatorname{median}}

\newtheorem{lemma}{Lemma}
\newtheorem{fact}{Fact}
\newtheorem{theorem}{Theorem}

\usepackage{xcolor}

\begin{document}

\title{\ifdefined\exab (Extended abstract) \fi
Improving Monte Carlo randomized approximation schemes}

\author{Mark Huber \\
Claremont McKenna College \\ {\tt mhuber@cmc.edu} }

\maketitle

\begin{abstract}
Consider a central problem in randomized approximation
schemes that use a Monte Carlo approach.  Given a sequence
of independent, identically distributed random variables 
$X_1,X_2,\ldots$ with mean $\mu$ and standard deviation at most
$c \mu$, where $c$ is a known constant, and $\epsilon,\delta > 0$,
create an estimate
$\hat \mu$ for $\mu$ such that 
$\prob(|\hat \mu - \mu| > \epsilon \mu) \leq \delta$.  This
technique has been used for building randomized approximation schemes
for the volume of a convex body, the permanent of a nonnegative matrix,
the number of linear extensions of a poset, the partition function of
the Ising model and many other problems.  Existing methods use (to the leading
order) $19.35 (c/\epsilon)^2 \ln(\delta^{-1})$ samples.  This is the 
best possible number up to the constant factor, 
and it is an open question as to 
what is the best constant possible.  This work gives an easy to apply 
estimate that only uses $6.96 (c/\epsilon)^2 \ln(\delta^{-1})$ samples in the 
leading order.
\end{abstract}

\thispagestyle{empty}

\section{Introduction}

The most common form of randomized approximation algorithm works by
finding an 
$X$ whose mean $\mu$ matches the true answer given the input, next
drawing $X_1,\ldots,X_k$ independently and identically distributed (iid)
according to $X$, and finally creating an estimate $\hat \mu$ for
$\mu$ as a function of these random samples.

Applications of this technique include finding the 
partition function of a Gibbs distribution~\cite{hubertoappearc},
approximating the number of linear extensions of a 
poset~\cite{huber2010d}, estimating the volume of a convex
body~\cite{lovaszv2006}, approximating the permanent of a nonnegative
matrix~\cite{jerrumsv2004}, approximating the 
normalizing constant for the ferromagnetic Ising model~\cite{jerrums1993},
finding the maximum likelihood for 
spatial point process models~\cite{huber2009c}, and many others.


Say that such an estimate $\hat \mu$ for $\mu$ is an
{\em $(\epsilon,\delta)$-randomized approximation scheme} 
(or $(\epsilon,\delta)$-ras) if
\begin{equation}
\prob(|\hat \mu - \mu| > \epsilon \mu) \leq \delta.
\end{equation}
That is, the chance that the absolute relative error in the estimate is 
greater than $\epsilon$ is at most $\delta$ in an 
$(\epsilon,\delta)$-ras.  The goal is to create
an $(\epsilon,\delta)$-ras using
$X_1,\ldots,X_T$, where $T$ is a random variable that has a small mean.

For each of 
the applications mentioned earlier, 
it was shown how to build random variables
$X_i$ such that $\sd(X_i) \leq c \mu$ for a known constant $c$ that
is an easily computable function of the input.  
Given this restriction on the standard deviation, it is well known
how to generate an $(\epsilon,\delta)$-ras using at most 
$19.35 c^2 \epsilon^{-2}\ln(\delta^{-1})$ samples (plus lower order terms).
The details are discussed further in
Section~\ref{SEC:estimating}.

On the other hand, it is known from an application of Wald's sequential 
ratio test that any such algorithm requires at least
$\Omega(c^2 \epsilon^{-2} \ln(\delta^{-1}))$
samples on average (as shown in~\cite{dagumklr2000}), 
therefore this is the best possible up to 
the constant factor.  The question remains of what is the best constant
factor.

This work introduces a simple new algorithm for estimating 
$\mu$ that reduces this constant from 19.35 to 6.96.  
For all of the applications
listed above, this approach immediately improves the constant of
the running time by a factor of 2.78.

First define a nuisance factor that will appear in several results.  Let
\begin{equation}
f(\epsilon) = (1 - \epsilon)^{-2}(1 + \epsilon - \epsilon^2)^{-1}(1 + \epsilon).
\end{equation}
Since $f(\epsilon) = 1 + 2 \epsilon + O(\epsilon^2)$, 
the leading order terms of $\epsilon^{-2}$ and $\epsilon^{-2} f(\epsilon)$
are identical.

\begin{theorem}
Suppose $X_1,X_2,\ldots$ is an independent, identically distributed
sequence of random variables with
mean $\mu$ and standard deviation at most $c \mu$, where $c$ is
a known constant.  Then for $\epsilon \in (0,1/3)$ and 
$\delta \in (0,1)$,
it is possible
to find $\hat \mu$ such that 
$\prob(|\hat \mu - \mu| > \epsilon \mu) \leq \delta$ using
$X_1,\ldots,X_t$, where 
\begin{equation}
t = \left\lceil \left(\frac{c}{\epsilon} \right)^2 f(\epsilon) \right\rceil
    \left(2 \left\lceil\frac{\ln(2\delta^{-1})}{\ln(4/3)} 
    \right\rceil + 1 \right).
\end{equation}
\end{theorem}

Note $2/\ln(4/3) \leq 6.96$, which gives the factor in
the leading order term mentioned earlier.

\section{Estimating $\mu$}
\label{SEC:estimating}

The classic estimate for $\mu$ uses sample averages of the $X_i$.  
Let $S_k = (X_1 + \cdots X_k)/k$.  Then $\mean[S_k] = \mu$ and 
$\sd(S_k) = \sd(X)/\sqrt{k}$.  
Say that the estimate {\em fails} if the absolute relative error 
$|(\hat \mu /\mu) - 1|$ is greater than $\epsilon$.
Using Chebyshev's
inequality gives 
$\prob(|(\hat \mu/\mu) - 1| > \epsilon) \leq c^2 /(\epsilon^{2}k).$
This gives a bound on the probability of failure that
only goes down polynomially in the number of samples $k$.

\subsection{Median of means}

A well known estimate with an 
exponentially small chance of failure goes back to at 
least~\cite{dyerf1991}.  The central idea is to look at the median
of several draws, each of which is the average of some fixed number
of draws of the original random variable.

Let $k = \lceil 8 (c/\epsilon)^2 \rceil$ draws of $X_i$.  So
Then $\mean[S_k] = \mu$ and $\sd(S_k) = \epsilon \mu/\sqrt{8}$.
Then Chebyshev's inequality gives
\begin{equation}
\prob(|S_k - \mu| \geq \epsilon \mu) \leq 1/8.
\end{equation}

The next step is to draw $W_1,\ldots,W_{2k + 1}$ iid from the same 
distribution as $S_k$.  Then it is highly likely that the median of
the $\{W_i\}$ values falls inside the region that has $7/8$ probability.

To be precise:
\begin{lemma}
\label{LEM:errorbound}
Suppose $\prob(R \leq a) \leq p$ and $\prob(R \geq b) \leq p$ for 
some $a < b$.  Then for $R_1,R_2,\ldots \sim R$ iid,
\begin{equation}
\prob(\median\{R_1,\ldots,R_{2k+1}\} \notin [a,b]) \leq 4 (\pi(k+1))^{-1/2}
  [4 p (1 - p)]^k.
\end{equation}
For $\prob(R \in [a,b]) \geq 1 - p$ then 
\begin{equation}
\prob(\median\{R_1,\ldots,R_{2k+1}\} \notin [a,b]) \leq 2 (\pi(k+1))^{-1/2}
  [4 p (1 - p)]^k.
\end{equation}
\end{lemma}

\ifdefined\exab 
[See the Appendix for the proof.]
\else 
\begin{proof}
  Suppose that $\prob(R \in [a,b]) \geq 1 - p$.  If $\prob(R \in [a,b]) = 1$
  the chance the median is not in $[a,b]$ is 0 and the inequality holds.

  Otherwise, 
  let $U$ be uniform over
  $[0,1]$, $R_{\text{in } [a,b]}$ 
  be the distribution of $R$ conditioned on $R \in [a,b]$,
  and $R_{\text{not in } [a,b]}$ 
  be the distribution of $R$ conditioned on $R \notin [a,b]$.

  Let $\ind(\text{expression})$ be the indicator function that is
  1 when the expression is true, and 0 when it is false.
  Then for independent random variables $U,$ $R_{\text{in} [a,b]}$, 
  $R_{\text{not in } [a,b]}$,
  \begin{equation*}
    R \sim R_{\text{in } [a,b]} \ind(U \leq \prob(R \in [a,b])) + 
    R_{\text{not in } [a,b]} \ind(U > \prob(R \in [a,b])).
  \end{equation*}

  With this representation, 
  the median of $2k + 1$ draws of $R$ will fall into $[a,b]$ if
  at least $k + 1$ of the $U_i$ fall into $[0,\prob(R \in [a,b])]$, which
  in turn occurs when the median of the $U_i$ is in $[0,1-p]$.  

  The median of the $U_i$ (call it $M$) is well known to have a beta
  distribution with density $f_M(x) = x^k(1 - x)^k\Gamma(2k+2)/\Gamma(k+1)^2$.
  Note $\prob(M > 1 - p) = \int_{1-p}^1 f_M(x) \ dx = 
  \int_{0}^p f_M(x) \ dx = \int_0^p [x(1 - x)]^k \Gamma(2k+2)/\Gamma(k+1)^2.$

  For $x \in [0,p]$, $x(1-x) \leq p^2 + (1 - 2p)x$, so 
  \begin{align*}
  \prob(M > 1- p) &\leq \int_0^p \frac{[p^2 + (1 - 2p)x]^k \Gamma(2k+1)}
     {\Gamma(k)^2} \\
   &= \frac{\Gamma(2k+2)}{\Gamma(k+1)^2} \cdot 
      \left[\frac{(p - p^2)^k - p^{2k}}{(k+1)(1 - 2p)}\right] \\
   &= \frac{(2k+2)!}{(k+1)!(k+1)!} \cdot \frac{k+1}{2k+2} \cdot 
      \left[\frac{p^k[(1 - p)^k - p^{k}]}{1 - 2p}\right]
  \end{align*}

  The factor $(2k+2)!/[(k+1)!]^2$ (known as a central binomial 
  coefficient) is well known to be at most $2^{2k+2}/\sqrt{\pi(k+2)}$  
  (see~\cite{koshy2008}).  Simplifying and neglecting the $-p^k$ term 
  then gives the result.
  
  The result where both $\prob(R \leq a) \leq p$ and $\prob(R \geq p)$
  is similar.  
\end{proof}
\fi

Applying this lemma to $W_i$ with probability $7/8$ of landing in the
desired location gives a chance of error at most 
$[4(7/8)(1 - 7/8)]^k = \exp(-\ln(7/16) k)$ for $k \geq 5$.  
Now the chance of failure is declining exponentially.

\begin{lemma}
The preceding procedure gives an $(\epsilon,\delta)$-ras
for $\mu$ that uses at most 
$\lceil 8 (c/\epsilon)^{2}\rceil 
 (2 \lceil \ln(\delta^{-1})/\ln(16/7) \rceil + 1)$ 
samples.
\end{lemma}

\begin{proof}
An instance of $W$ takes $\lceil 8 (c/\epsilon)^{2} \rceil$ draws
from $X$ to produce.

To make $[4(7/8)(1-7/8)]^k \leq \delta$, 
$k \geq \ln(\delta^{-1})/\ln(16/7)$.  Since
$2k + 1$ draws of $W$ are necessary, the result follows.
\end{proof}

The method just described could be done with 
$W \sim S_{\lceil i(c/\epsilon)^2 \rceil}$ for any
$i$.  The choice of $i = 8$ minimizes the constant in the running time
given in the previous lemma.  In finding $k$, the $4(\pi(k+1))^{-1/2}$
factor was bounded by 1 for $k \geq 5$.  Using the full factor leaves
the first order term unchanged, only affecting lower order terms.

Note $8\cdot (1/\ln(16/7)) \cdot 2 \approx 19.35$, giving the constant
mentioned earlier.

\subsection{The new estimate}

The new method creates a new random variable $V$ such that 
\begin{equation}
\prob(|V - \mean[V]| \geq \epsilon \mu) \leq 1/4,
\end{equation}
but only using slightly more than $(c/\epsilon)^2$ 
draws from $X_i$.  To accomplish the same feat using Chebyshev's 
inequality would require $4(c/\epsilon)^2$ draws from the $X_i$.

Before describing the procedure, it will help to have an understanding
of why a random variable does not always lie inside the standard deviation.  
Suppose that $Y$ has 
mean $\mu$ and standard deviation $\epsilon \mu$.  An example of such
a random variable is 
$\prob(Y = \mu - \epsilon \mu) = \prob(Y = \mu + \epsilon \mu) = 1/2$.

Note that $\prob(Y < \mu + \epsilon \mu) = 1/2$, because fully half of the 
probability is sitting just outside the interval 
$(\mu - \epsilon \mu,\mu + \epsilon \mu)$.  The goal is to create
$V$ such that $\prob(V < \mu + \epsilon \mu) \geq 3/4$.

To create such a $V$, let $R$ be 
uniform over the interval $[1 - \epsilon,1 + \epsilon]$, and independent
of $Y$.  Then set $V = RY$.  So
\begin{equation}
\prob(YR < \mu + \epsilon \mu) = \frac{1}{2} + \frac{1}{2} \prob(R < 1)
 = \frac{3}{4}
\end{equation}
and
\begin{equation}
\prob(YR > \mu - \epsilon \mu) = \frac{1}{2} + \frac{1}{2} \prob(R > 1)
 = \frac{3}{4}.
\end{equation}

So for this particular $Y$, by applying this simple random scaling procedure,
the chance of falling into the tails has a bound as small as 
when $4$ samples from $Y$ are averaged together!

This is the idea behind the new estimate.
Given $\epsilon > 0$, for
a given draw $S_i \sim S$, draw 
$R_i \sim \unifdist([1-\epsilon,1+\epsilon])$
independently of $S_i$.  Next, let
\begin{equation}
W'_i = S_i R_i.
\end{equation}

This smoothed random variable still has $\mean[W'_i] = \mu$, moreover,
it is now much more likely to lie within a standard deviation of its mean!

\begin{lemma}
\label{LEM:failure}
Let $S$ have mean $\mu$ and standard deviation at most
$\epsilon \mu / \sqrt{f(\epsilon)}$, where $\epsilon \leq 1/3$.
Let $R$ be independent of $S$ and uniform over 
$[1-\epsilon,1+\epsilon]$.  Then
\begin{equation}
\prob(S R \leq \mu - \epsilon \mu) \leq \frac{1}{4}, \quad
\prob(S R \geq \mu + \epsilon \mu) \leq \frac{1}{4}.
\end{equation}
\end{lemma}

This lemma is the heart of the new estimate, and 
will be proved in the next section.  Using 
this lemma together with Lemma~\ref{LEM:errorbound} immediately
gives the following result.

\begin{lemma}
\label{LEM:main}
Fix $\epsilon \in (0,1/3)$, $t$ a positive integer, 
and let $X$ be a random variable
with mean $\mu$ and standard deviation at most $c\mu$.  
For $i = 1,2,\ldots,t$, let $S_i$ be the sample
average of $\lceil (c/\epsilon)^2 f(\epsilon)\rceil$ 
iid draws from the distribution
of $X$.  Independently, draw 
$R_1,\ldots,R_t$ uniformly from $[1-\epsilon,1+\epsilon]$.  
Let
\begin{equation}
t = \lceil (cf(\epsilon)/\epsilon)^{2}\rceil 
 (2 \lceil \ln(2\delta^{-1})/
 \ln(4/3) \rceil + 1).
\end{equation}
Then it holds that 
\begin{equation}
\prob(|\median\{S_1R_1,\ldots,S_tR_t\} - \mu| > \epsilon \mu) \leq \delta.
\end{equation}
\end{lemma}

\section{Proof of Lemma~\ref{LEM:failure}}

The proof of Lemma 3 takes the following steps.  Let
$S$ be a random variable with mean $\mu$ and standard deviation at
most $\alpha \mu$.  The goal is to show that for 
$R \sim \unifdist([1-\epsilon,1+\epsilon])$, 
$\prob(SR \leq \mu - \epsilon \mu)$ and 
$\prob(SR \geq \mu + \epsilon \mu)$ are both at most $1/4$.
\begin{enumerate}
\item{Eliminate the $\mu$ factor by considering $Y = S/\mu$ 
so $\mean[Y] = 1$ and $\sd(Y) \leq \alpha$.  The new
goal is to upper bound $q_{+} = \prob(YR \geq 1 + \epsilon)$ and 
$q_{-} = \prob(YR \leq 1 - \epsilon)$.}
\item{First look at $q_{+}$.  
      Write $Y$ as a mixture of random variables 
      $Y_1$ and $Y_2$ where there is a value
      $y$ such that $\prob(Y_1 \leq y) = \prob(Y_2 > y) = 1$. 
      (Lemma~\ref{LEM:mixture}.)}
\item{Show that for a proper choice of $y$, 
      the value $q_{+}$ is maximized when $Y_1$ and $Y_2$ are 
      not random, but deterministic functions of $\epsilon$, and 
      when the standard deviation of $Y$  equals the upper bound $\alpha \mu$.
      (Lemmas~\ref{LEM:points} and~\ref{LEM:points2}.)} 
\item{Show that for $\sd(Y) = \epsilon/\sqrt{f(\epsilon)}$, 
      $q_{+}$ is at most $1/4$ 
      given that $Y_1$ and $Y_2$ are each concentrated at
      a single value. (Lemma~\ref{LEM:f2}.)}
\item{The proof for $q_{-}$ is then accomplished in a similar fashion.}
\end{enumerate}

To begin, note that 
if $Y < 1$, then $YR < 1 + \epsilon$ always.
When $Y > (1 + \epsilon)/(1 - \epsilon)$, 
then $YR > 1 + \epsilon$.  Then there is the case in the middle, where
$Y \in [1,(1 + \epsilon)/(1 - \epsilon)]$ and it could be true that
$YR \leq 1 + \epsilon$ if $R$ is small enough.  For that reason, define
the intervals $I_1 = (-\infty,1)$,
$I_2 = [1,(1+\epsilon)/(1-\epsilon)]$, 
$I_3 = ((1+\epsilon)/(1 - \epsilon),\infty)$.  

Write $Y$ as a mixture of two random variables,
one of which has support on $I_1$, and the other has support on $I_2 \cup I_3$.
\begin{lemma}
\label{LEM:mixture}
For any random variable $Y$, there exist $Y'$, $C$, $Y_1$, and $Y_2$ 
such that $Y' \sim Y$ and
\begin{equation}
Y' = C Y_1 + (1 - C) Y_2,
\end{equation}
where $Y_1$ falls into $I_1$ with probability 1, $Y_2$ falls into
$I_2 \cup I_3$ with probability 1, and $C$ is a Bernoulli random variable
with
$\prob(C = 1) = \prob(Y \in I_1)$ and
$\prob(C = 0) = \prob(Y \in I_2 \cup I_3)$.
\end{lemma}

\begin{proof}
If either $\prob(Y \in I_1)$ or $\prob(Y \in I_2 \cup I_3)$ equal 1, the
result is trivial, otherwise, let $Y_1 \sim [Y|Y \in I_1]$,
$Y_2 \sim [Y|Y \in I_2 \cup I_3]$, and the result follows.
\end{proof}

Since $Y' \sim Y$, 
$\prob(YR \leq 1 + \epsilon) = \prob(Y'R \leq 1 + \epsilon)$.
So without loss of generality work with $Y'$ from here on out.

Recall that the probability that a event occurs is just the 
expected value of the indicator function of that event.  So
\begin{equation}
\prob(Y'R \leq 1 + \epsilon) = \mean[\ind(Y'R \leq 1 + \epsilon)].
\end{equation}
We also need the following well known fact about conditional expectation
(see for instance~\cite{durrett2010}).

\begin{fact}
If $A$ and $B$ are random variables such that $A$ and $[A|B]$
are both integrable, then
\begin{equation}
\mean[A] = \mean[\mean[A|B]].
\end{equation}
\end{fact}

Hence
\begin{equation}
\prob(Y'R \leq 1 + \epsilon) = 
  \mean[\mean[\ind(Y'R \leq 1 + \epsilon)|Y']
\end{equation}
and our analysis can start with 
the inside expectation $\mean[\ind(Y'R \leq 1 + \epsilon)|Y']$.

\begin{lemma}
\begin{equation}
\mean[\ind(Y'R \leq 1 + \epsilon|Y'] 
 = \ind(Y' \in I_1) + \ind(Y' \in I_2)
   [((1 + \epsilon)/Y') - (1 - \epsilon)](2\epsilon)^{-1}.
\end{equation}
\end{lemma}

\begin{proof}
Let $f(R,Y') = \ind(R \leq (1 + \epsilon)/Y')$.  Then 
\begin{equation}
f(R,Y') = f(R,Y') \ind(Y' \in I_1) + 
 f(R,Y') \ind(Y' \in I_2) + f(R,Y') \ind(Y' \in I_3).
\end{equation}
As noted earlier, $f(R,Y') \ind(Y' \in I_3) = 0$ and 
$f(R,Y') \ind(Y' \in I_1) = \ind(Y' \in I_1).$  That means
\begin{equation}  
\mean[f(R,Y')|Y'] = \mean[\ind(Y' \in I_1)|Y']
 + \mean[\ind(Y' \in Y_2) f(R,Y')|Y'].
\end{equation}

Since $\ind(Y' \in I_2)$ is measurable with respect to $Y'$,
$\mean[\ind(Y' \in I_1)|Y'] = \ind(Y' \in I_1).$  In the expectation
$\mean[\ind(Y' \in I_2)f(R,Y')|Y']$, treat $Y'$ as a constant.  Since
$R$ is uniform over $[1-\epsilon,1+\epsilon]$, the chance that it
is at most $(1 - \epsilon)/Y'$ when $Y' \in I_2$ is
$[[(1 + \epsilon)/Y'] - (1 - \epsilon)]/[(1 + \epsilon) - (1 - \epsilon)]$.
Hence
\begin{equation}
\mean[f(R,Y')|Y'] = \ind(Y' \in I_1) + \ind(Y' \in I_2)
 [((1 + \epsilon)/Y')  - (1 - \epsilon)](2\epsilon)^{-1}.
\end{equation}
\end{proof}

\begin{lemma}
Let $p = \prob(Y' \in I_1)$.  Then
\begin{equation}
\label{EQN:p}
\mean[\ind(RY' \leq 1 + \epsilon)]
 = p + (1 - p)
 \mean\left[\ind(Y_2 \in I_2)\frac{((1 + \epsilon)/Y_2) 
  - (1 - \epsilon)}{2\epsilon}\right].
\end{equation}
\end{lemma}

\begin{proof}
From the last lemma
\[
\mean[f(R,Y')] 
 = \mean[\ind(Y' \in I_1)] + 
   \mean[\ind(Y' \in I_2)
 [((1 + \epsilon)/Y')  - (1 - \epsilon)](2\epsilon)^{-1}].
\]

Now $\mean[\ind(Y' \in I_2)] = \prob(Y' \in I_2) = p$.  For the 
second term,
\begin{align*}
\mean&[\ind(Y' \in I_2)
 [((1 + \epsilon)/Y')  - (1 - \epsilon)](2\epsilon)^{-1}] \\
 &= \mean[\mean[\ind(Y' \in I_2)
 [((1 + \epsilon)/Y')  - (1 - \epsilon)](2\epsilon)^{-1}|C]] \\
 &= \prob(C = 1)(0) + \prob(C = 0)
    \mean[\ind(Y_2 \in I_2)
 [((1 + \epsilon)/Y_2)  - (1 - \epsilon)](2\epsilon)^{-1}|C]].
\end{align*}
Using $\prob(C = 0) = 1 - p$ completes the proof.
\end{proof}

Recall Jensen's inequality.
\begin{fact}[Jensen's inequality]
If $X$ is a random variable with finite mean, $\prob(X \in A) = 1$, and 
$g$ is a convex measurable function over $A$, then
\[
\mean[g(X)] \geq g(\mean[X]).
\]
\end{fact}

Now $g(x) = \ind(x \in I_2)((1+\epsilon)/x - (1 - \epsilon)(2\epsilon)^{-1}$
is a convex function over all $x \in I_2 \cup I_3$.  Similarly,
$g(x) = 1$ is a convex function over all $x \in I_1$.  That gives
the following.

\begin{lemma}
\label{LEM:points}
Suppose $Y'' = C \mean[Y_1] + (1 - C)\mean[Y_2]$.  Then
$\mean[Y''] = \mean[Y'] = 1$ and $\sd(Y'') \leq \sd(Y')$.  Also
\begin{equation}
\label{EQN:better}
\prob(Y'' R \leq 1 + \epsilon) \leq \prob(Y' R \leq 1 + \epsilon).
\end{equation}
\end{lemma}

\begin{proof}
Equation~\eqref{EQN:better} follows immediately from 
Jensen's inequality applied to~\eqref{EQN:p}.  
The statement $\mean[Y''] = \mean[Y']$ follows
from $\mean[Y'] = \mean[\mean[Y'|C]].$ 
Replacing a component of a mixture by its expectation gives a 
standard deviation at most that of the original random variable, hence
$\sd(Y'') \leq \sd(Y').$
\end{proof}

So there exist $a_1$ and $a_2$ such that 
$Y'' \in \{1 + a_1,1 + a_2\}$, where $a_1 \leq 0$ and $a_2 > 0$ (since
$\mean[Y''] = 1$).  Note that 
$\sd(Y'') = \prob(Y'' = a_1)a_1^2 + \prob(Y'' = a_2)a_2^2$.

\begin{lemma}
\label{LEM:points2}
Let ${\cal Y}$ be the set of random variables with mean $1$ and 
standard deviation at most $\alpha$.  Let
${\cal W}$ be the set of random variables $W$ with mean $1$, standard
deviation equal to $\alpha$, and 
$W$ takes on one of two values with
probability 1.  Then
\[
\min_{Y \in {\cal Y}} \prob(YR \leq 1 + \epsilon) \geq
 \min_{W \in {\cal W}} \prob(WR \leq 1 + \epsilon).
\]
\end{lemma}

\begin{proof}
Let ${\cal Y}''$ be the set of mean 1, standard deviation 
at most $\alpha$ random variables that only take on one of two values 
with probability 1.  Then the previous lemma got us to
\[
\min_{Y \in {\cal Y}} \prob(YR \leq 1 + \epsilon) \geq
 \min_{Y'' \in {\cal Y}''} \prob(Y''R \leq 1 + \epsilon).
\]

Suppose $Y'' \in {\cal Y}''$ with 
$p_1 = \prob(Y'' = 1 + a_1)$ and $p_2 = 1 - p_1 = \prob(Y'' = 1 + a_2)$, where
$a_1 \leq 0$ and $a_2 > 0$ and $p_1,p_2 \geq 0$.  With this notation,
$\sd(Y'') = p_1 a_1^2 + p_2 a_2^2$, and 
\[
\prob(Y''R \leq 1 + \epsilon) = p_1 + 
 p_2(\ind(a_2 \in [0,2\epsilon/(1-\epsilon)])\left(\frac{((1+\epsilon)/a_2) - 
 (1 - \epsilon)}{2\epsilon}\right).
\]
This is a decreasing function of $a_2$ and independent of $a_1$.

So construct $W$ by setting $k_1 = \alpha a_1/\sd(Y'')$.
$k_2 = \alpha a_2/\sd(Y'').$  Since $\alpha \geq \sd(Y'')$, 
$k_2 \geq a_2$, which means 
$\prob(W R \leq 1 + \epsilon) \leq \prob(Y'' R \leq 1 + \epsilon)$.
Note $\mean[W] = 1$ and $\sd(W) = \epsilon$, proving the 
result.
\end{proof}

What has been accomplished so far it to show that 
$\prob(Y R \leq 1 + \epsilon)$ is at least the optimal objective function
value for the optimization
problem:
\begin{align*}
\min{} &p_1 + p_2(\ind(k_2 \in [0,2\epsilon/(1-\epsilon)])
 ((1+\epsilon/k_2) - (1 - \epsilon))(2\epsilon)^{-1} \\
\text{subject to } & p_1 + p_2 = 1 \\
                   & p_1 k_1 + p_2 k_2 = 0 \\
                   & p_1 k_1^2 + p_2 k_2^2 = \alpha
\end{align*}
Using the constraints to solve for $p_1$ and $p_2$ in terms of $k_2$ turns
the objective function into $\min_{k_2 > 0} h(k_2)$, where
\begin{equation}
\label{EQN:h}
h(k_2)= \frac{k_2^2}{k_2^2 + \alpha^2} + 
  \frac{\alpha^2}{k_2^2 + \alpha^2}
 \ind\left(k_2 \leq \frac{2\epsilon}{1-\epsilon}\right)
 \frac{(1+\epsilon)/(1+k_2) - (1 - \epsilon)}{2\epsilon}.
\end{equation}

\begin{lemma}
\label{LEM:f1}
For $\alpha = \epsilon \sqrt{
   [(1-\epsilon)^2(1 + \epsilon - \epsilon^2)]/[1 + \epsilon]}$
the minimum of $h(k_2)$ for $k_2 > 0$ is at least 3/4.
\end{lemma}

\begin{proof}
When $k_2 \geq 2\epsilon/(1 - \epsilon)$, $h(k_2) = k_2^2/(k_2^2 + \alpha^2)$,
which is an increasing function of $k_2$.  Therefore the minimum occurs
at $k_2 = 2\epsilon/(1 - \epsilon)$ so it is only necessary to 
consider $k_2 \in (0,2\epsilon/(1 - \epsilon)]$.

Assuming this to be true and simplifying $h(k_2)$ gives
\begin{equation}
h(k_2) = 1 + \frac{\alpha^2}{k_2^2 + \alpha^2}
 \left[
 \left(\frac{1 + \epsilon}{1 + k_2} - (1 - \epsilon)\right)(2\epsilon)^{-1}
 - 1 \right].
\end{equation}

The right hand side is at least 3/4 if and only if
\begin{equation}
\alpha^2 \leq f_1(k_2) \text{ where } 
 f_1(k_2) = \frac{k_2^2(k_2 + 1)\epsilon}{k_2(\epsilon + 2) - \epsilon}. 
\end{equation}

Now
\begin{equation*}
\frac{df_1(k_2)}{dk_2} = 
  \frac{2 k_2 \epsilon(k_2^2(\epsilon + 2) + k_2(1 - \epsilon) - \epsilon)}
 {[k_2(\epsilon+2)-\epsilon]^2}.
\end{equation*}
The denominator and the $2k_2 \epsilon$ factor in the numerator is
always positive, leaving only the quadratic factor 
$k_2^2(\epsilon+2) + k_2(1 - \epsilon) - \epsilon$ to determine the 
sign.  This factor starts
negative at $k_2 = 0$, and is then increasing, so the minimum of 
$f_1(k_2)$ occurs when the quadratic is zero, which happens
at 
$k^*_2(\epsilon) = [\epsilon - 1 + \sqrt{5\epsilon^2 + 6\epsilon + 1}]/(2(\epsilon + 2)).$  It is easy to bound this expression for $\epsilon \in (0,1)$:
\begin{equation}
\epsilon - \epsilon^2 \leq k^*_2(\epsilon) \leq \epsilon.
\end{equation}

Using the lower bound on $k_2^*$ for the numerator and the upper bound
on $k_2^*$ for the denominator of $f_1(k_2)$ gives
\begin{equation}
f_1(k_2) \geq \frac{(\epsilon - \epsilon^2)^2(1 + \epsilon - \epsilon^2)\epsilon}
  {\epsilon(\epsilon + 2) - \epsilon}
 = \epsilon^2
   \frac{(1-\epsilon)^2(1 + \epsilon - \epsilon^2)}{1 + \epsilon}
 = \epsilon^2 / f(\epsilon).
\end{equation}
This is why $f(\epsilon)$ was chosen to be what it is.
Therefore, for $\alpha^2 \leq \epsilon^2/f(\epsilon)$, $\alpha^2 \leq f_1(k_2)$,
and 
$h(k_2) \geq 3/4$.
\end{proof}

This lemma shows that $\prob(SR \leq (1 + \epsilon)\mu) \leq 3/4$
for all $S$ with mean $\mu$ and standard deviation at most 
$\epsilon \mu / f(\epsilon)$.
This same sequence of steps, where we first show that we need
only consider random variables that take on two values, and then 
use the constraints on the variable to reduce it to a one dimensional
optimization problem, and then finally obtain a bound on the standard
deviation.  

The result is a function similar to $h$ of~\eqref{EQN:h}.  Define
\begin{equation}
h_2(k_1) = \frac{k_1^2}{k_1^2 + \alpha^2}
 + \frac{\alpha^2}{k_1^2 + \alpha^2}
   \ind\left(k_1 \geq \frac{-2\epsilon}{1 - \epsilon}\right)
 \frac{1 + \epsilon - (1 - \epsilon)/(1 + k_1)}{2\epsilon}.
\end{equation}
Then $\prob(SR \geq (1 - \epsilon)\mu) \geq \min_{k_1 \leq 0} h_2(k_1)$.

\begin{lemma}
\label{LEM:f2}
For $\alpha = \epsilon\sqrt{[(1+\epsilon)^3(1-2\epsilon)]
      / [1-3\epsilon+2\epsilon^2]}$, 
the minimum of $h_2(k_1)$ for $k_1 \leq 0$ is at least $3/4$.
\end{lemma}

The proof is similar to that of Lemma~\ref{LEM:f1}.

For $\epsilon \leq 1/3$, 
the bound on $\alpha$ from Lemma~\ref{LEM:f1} is stronger than the
bound from Lemma~\ref{LEM:f2}.  This proves Lemma~\ref{LEM:failure}.

\section{Summary}

For any random variable with $\mean[S] = \mu$ and 
$\sd(S) \leq \epsilon/f(\epsilon)$, simply generating independently $R$
uniformly over $[1-\epsilon,1+\epsilon]$ gives $SR$ with the 
properties that 
$\mean[SR] = \mu$, $\prob(SR - \mu \geq \epsilon \mu) \leq 1/4$, and
$\prob(SR - \mu \leq \epsilon \mu) \leq 1/4$.  
Since generating $S$ is usually very costly (as 
in~\cite{hubertoappearc,huber2010d,lovaszv2006,jerrumsv2004,jerrums1993,huber2009c}), generating $R$
incurs very little overhead and then allows the median of 
independent draws to quickly concentrate the resulting approximation.

\bibliographystyle{plain}
\bibliography{../../../2014/2014_refs}

\begin{thebibliography}{10}

\bibitem{huber2010d}
J.~Banks, S.~Garrabrant, M.~Huber, and A.~Perizzolo.
\newblock Using {TPA} for approximating the number of linear extensions.
\newblock arXiv:1010.4981. Submitted, 2010.

\bibitem{dagumklr2000}
P.~Dagum, R.~Karp, M.~Luby, and S.~Ross.
\newblock An optimal algorithm for {M}onte {C}arlo estimation.
\newblock {\em Siam. J. Comput.}, 29(5):1484--1496, 2000.

\bibitem{durrett2010}
R.~Durrett.
\newblock {\em Probability: Theory and Examples, 4th edition}.
\newblock Cambridge University Press, 2010.

\bibitem{dyerf1991}
Martin~E. Dyer and Alan~M. Frieze.
\newblock Computing the volume of a convex body: A case where randomness
  provably helps.
\newblock In B{\'e}la Bollob{\'a}s, editor, {\em Proceedings of AMS Symposium
  on Probabilistic Combinatorics and Its Applications}, volume~44 of {\em
  Proceedings of Symposia in Applied Mathematics}, pages 123--170. American
  Mathematical Society, 1991.

\bibitem{hubertoappearc}
M.~Huber.
\newblock Approximation algorithms for the normalizing constant of {G}ibbs
  distributions.
\newblock {\em Ann. Appl. Probab.}
\newblock {a}rXiv:1206.2689. To appear.

\bibitem{huber2009c}
M.~L. Huber and R.~L. Wolpert.
\newblock Likelihood-based inference for {M}at\'{e}rn type-{III} repulsive
  point processes.
\newblock {\em Adv. Appl. Prob.}, 41(4):958--977, 2009.

\bibitem{jerrums1993}
M.~Jerrum and A.~Sinclair.
\newblock Polynomial-time approximation algorithms for the {I}sing model.
\newblock {\em SIAM J. Comput.}, 22:1087--1116, 1993.

\bibitem{jerrumsv2004}
M.R. Jerrum, A.~Sinclair, and E.~Vigoda.
\newblock A polynomial-time approximation algorithm for the permanent of a
  matrix with nonnegative entries.
\newblock {\em J. of the ACM}, 51(4):671--697, 2004.

\bibitem{koshy2008}
T.~Koshy.
\newblock {\em Catalan Numbers with applications}.
\newblock Oxford University Press, 2008.

\bibitem{lovaszv2006}
L.~Lov\'asz and S.~Vempala.
\newblock Simulated annealing in convex bodies and an $o^*(n^4)$ volume
  algorithm.
\newblock {\em J. Comput. Syst. Sci}, 72(2):392--417.

\end{thebibliography}

\ifdefined\exab
\section*{Appendix}

The full version of the paper follows.
\fi

\end{document}